\documentclass[11pt, twoside]{article}
\usepackage{latexsym}
\usepackage{amsmath}
\usepackage{amssymb}
\usepackage[all]{xy}
\usepackage{amsfonts}
\usepackage{verbatim}
\usepackage{amsthm}
\usepackage{mathrsfs}\usepackage{array}
\usepackage{epsfig}
\usepackage{xy}
\usepackage{tensor}
\usepackage{array}
\usepackage{stmaryrd}
\usepackage{graphicx,color}
\usepackage{xcolor}
\usepackage[colorlinks=true,linkcolor=blue,citecolor=blue]{hyperref}
\usepackage{tikz}
\usetikzlibrary{arrows,calc}
\usepackage{etex}
\usepackage{mathdots}
\usepackage{float}
\usepackage{graphics}
\usepackage{pdflscape}
\usepackage{extarrows}
\usepackage{anysize,hyperref}
\input xypic
\xyoption{all}
\usepackage{bm}
\usepackage[perpage,symbol]{footmisc}
\usepackage{setspace}
\SelectTips{cm}{10}

\def\C{\mathscr{C}}

\def\I{\mathcal{I}}
\def\P{\mathcal{P}}
\def\E{\mathbb{E}}
\def\s{\mathfrak{s}}
\def\Id{\mathrm{Id}}

\def\del{\delta}
\def\dr{\ar@{->}[r]}
\def\id{\mbox{\rm id}\,}\def\Im{\mbox{\rm Im}\,}\def\Ker{\mbox{\rm Ker}\,}

\newcommand{\coresdim}{{\rm coresdim}}
\newcommand{\resdim}{{\rm resdim}}

\def\X{\mathscr{X}}

\def\add{\mbox{add}}

\def\Hom{\mbox{Hom}}

\begin{document}
\baselineskip=15pt
\title{\Large{\bf  $\bm{n}$-cotorsion pairs and recollements of extriangulated categories
\footnotetext{$^\ast$Corresponding author.  ~Jian He is supported by the National Natural Science Foundation of China (Grant No. 12171230) and Youth Science and Technology Foundation of Gansu Provincial (Grant No. 23JRRA825). Jing He is supported by the Hunan Provincial Natural Science Foundation of China (Grant No. 2023JJ40217).} }}
\medskip
\author{Jian He and Jing He$^\ast$}

\date{}

\maketitle
\def\blue{\color{blue}}
\def\red{\color{red}}

\newtheorem{theorem}{Theorem}[section]
\newtheorem{lemma}[theorem]{Lemma}
\newtheorem{corollary}[theorem]{Corollary}
\newtheorem{proposition}[theorem]{Proposition}
\newtheorem{conjecture}{Conjecture}
\theoremstyle{definition}
\newtheorem{definition}[theorem]{Definition}
\newtheorem{question}[theorem]{Question}
\newtheorem{notation}[theorem]{Notation}
\newtheorem{remark}[theorem]{Remark}
\newtheorem{remark*}[]{Remark}
\newtheorem{example}[theorem]{Example}
\newtheorem{example*}[]{Example}

\newtheorem{construction}[theorem]{Construction}
\newtheorem{construction*}[]{Construction}

\newtheorem{assumption}[theorem]{Assumption}
\newtheorem{assumption*}[]{Assumption}

\baselineskip=17pt
\parindent=0.5cm

\begin{abstract}
\baselineskip=16pt
In this article, we prove that if  $(\mathcal A ,\mathcal B,\mathcal C)$ is  a recollement of extriangulated categories, then $n$-cotorsion pairs in $\mathcal A$ and $\mathcal C$ can induce $n$-cotorsion pairs in $\mathcal B$. Conversely, this holds true under natural assumptions.  Besides, we give mild conditions on a pseudo cluster tilting subcategory on the middle category of a recollement of extriangulated categories,  for the corresponding additive quotients to form a recollement of semi-abelian categories.\\[2mm]
\textbf{Keywords:} $n$-cotorsion pair; recollement; extriangulated categories. \\[2mm]
\textbf{ 2020 Mathematics Subject Classification:} 18G80; 18E10; 18E40
\end{abstract}
\medskip

\pagestyle{myheadings}
\markboth{\rightline {\scriptsize Jian He and Jing He\hspace{2mm}}}
         {\leftline{\scriptsize  $n$-cotorsion pairs and recollements of extriangulated categories}}

\section{Introduction}
The concept of a recollement in triangulated categories was first introduced by Beilinson, Bernstein and Deligne, as detailed in \cite{BBD}. A fundamental example of a recollement situation of abelian categories appeared in the construction of perverse sheaves by MacPherson and Vilonen \cite{MV}.  Recollements, both in abelian and triangulated categories, play a pivotal role in various domains such as ring theory, representation theory, and the geometry of singular spaces.

Motivated by certain properties exhibited by Gorenstein projective and Gorenstein injective modules over an Iwanaga-Gorenstein ring, Huerta, Mendoza, and P\'{e}rez \cite{HMP} introduced the concept of left $n$-cotorsion (resp. right $n$-cotorsion, $n$-cotorsion) pairs in an abelian category $\mathcal{C}$. 
It's noteworthy that $1$-cotorsion pairs coincide with the concept of complete cotorsion pairs. Let $(\mathcal{A}, \mathcal{B}, \mathcal{C})$ be a recollement of abelian categories. Recently, Cao, Wei and Wu \cite{CWW} described how to glue together $n$-cotorsion pairs in $\mathcal A$ and $\mathcal C$ to obtain an $n$-cotorsion pair in $\mathcal B$.

Nakaoka and Palu introduced the concept of extriangulated categories in their seminal work \cite{NP}. This notion simultaneously generalizes exact categories and triangulated categories. Exact categories, which include abelian categories, and extension-closed subcategories within extriangulated categories are considered as specific instances of extriangulated categories. Moreover, there exist additional instances of extriangulated categories that do not fall within the categories of exact or triangulated, as documented in \cite{NP,ZZ,HZZ}.
Wang, Wei, and Zhang \cite{WWZ} introduced the recollement of extriangulated categories, which is a simultaneous generalization of recollements of abelian categories and triangulated categories. They also provided conditions under which the glued pair with respect to cotorsion pairs in $\mathcal{A}$ and $\mathcal{C}$ forms a cotorsion pair in $\mathcal{B}$ for a recollement $(\mathcal{A},\mathcal{B},\mathcal{C})$ of extriangulated categories.
He and Zhou \cite{HZ} introduced $n$-cotorsion pairs in an extriangulated category with enough projectives and enough injectives. They demonstrated a one-to-one correspondence between $n$-cotorsion pairs and $(n+1)$-cluster tilting subcategories.
He and He \cite{HH} introduced pseudo cluster tilting subcategories $\mathcal{X}$ in an extriangulated category $\mathcal{C}$. They established that the quotient category $\mathcal{C}/\mathcal{X}$, obtained by dividing an extriangulated category by a pseudo cluster tilting subcategory, becomes a semi-abelian category. Moreover, they showed that $\mathcal{C}/\mathcal{X}$ achieves the status of an abelian category if and only if specific self-orthogonal conditions are satisfied.

Motivated by this, we consider to extend and cover relevant results of Cao, Wei and Wu \cite{CWW}. Suppose that $\mathcal B$ admits a recollement relative to extriangulated categories $\mathcal A$ and $\mathcal C$.
Our first main result describes how to glue together $n$-cotorsion pairs $(\mathcal T_1,\mathcal F_1)$ in $\mathcal A$ and $(\mathcal T_2,\mathcal F_2)$ in $\mathcal C$, to obtain an $n$-cotorsion pair
$(\mathcal T,\mathcal F)$ of $\mathcal B$, see Theorem \ref{main}. This extends their results of Cao, Wei and Wu \cite{CWW} in the framework of extriangulated categories. In the reverse direction, our second main result gives sufficient conditions on an $n$-cotorsion pair $(\mathcal T,\mathcal F)$ of $\mathcal B$, relative to the functors involved in the recollement, to induce $n$-cotorsion pairs in $\mathcal A$ and $\mathcal B$, see Theorem \ref{main1}.  Our third main result constructs a recollement
of semi-abelian categories from a recollement of extriangulated categories, see Theorem \ref{zh}.

This article is organized as follows. In Section 2, we give some terminologies
and some preliminary results. In Section 3, we prove our first and second main
results. In Section 4, we prove our third main
result.

\section{Preliminaries}
We briefly recall some definitions and basic properties of extriangulated categories from \cite{NP}.
We omit some details here, but the reader can find them in \cite{NP}.

Let $\mathcal{C}$ be an additive category equipped with an additive bifunctor
$$\mathbb{E}: \mathcal{C}^{\rm op}\times \mathcal{C}\rightarrow {\rm Ab},$$
where ${\rm Ab}$ is the category of abelian groups. For any objects $A, C\in\mathcal{C}$, an element $\delta\in \mathbb{E}(C,A)$ is called an $\mathbb{E}$-extension.
Let $\mathfrak{s}$ be a correspondence which associates an equivalence class $$\mathfrak{s}(\delta)=\xymatrix@C=0.8cm{[A\ar[r]^x
 &B\ar[r]^y&C]}$$ to any $\mathbb{E}$-extension $\delta\in\mathbb{E}(C, A)$. This $\mathfrak{s}$ is called a {\it realization} of $\mathbb{E}$, if it makes the diagrams in \cite[Definition 2.9]{NP} commutative.

A triplet $(\mathcal{C}, \mathbb{E}, \mathfrak{s})$ is called an {\it extriangulated category} if it satisfies the following conditions.
\begin{itemize}
\item $\mathbb{E}\colon\mathcal{C}^{\rm op}\times \mathcal{C}\rightarrow \rm{Ab}$ is an additive bifunctor.

\item $\mathfrak{s}$ is an additive realization of $\mathbb{E}$.

\item $\mathbb{E}$ and $\mathfrak{s}$  satisfy the compatibility conditions $(\rm ET3),(\rm ET3)^{\rm op},(\rm ET4),(\rm ET4)^{\rm op}$ in \cite[Definition 2.12]{NP}.
\end{itemize}

We collect the following terminology from \cite{NP}.

\begin{definition}
Let $(\mathcal{C},\E,\s)$ be an extriangulated category.
\begin{itemize}
\item[(1)] A sequence $A\xrightarrow{~x~}B\xrightarrow{~y~}C$ is called a {\it conflation} if it realizes some $\E$-extension $\del\in\E(C,A)$.
    In this case, $x$ is called an {\it inflation} and $y$ is called a {\it deflation}.

\item[(2)] If a conflation  $A\xrightarrow{~x~}B\xrightarrow{~y~}C$ realizes $\delta\in\mathbb{E}(C,A)$, we call the pair $( A\xrightarrow{~x~}B\xrightarrow{~y~}C,\delta)$ an {\it $\E$-triangle}, and write it in the following way.
$$A\overset{x}{\longrightarrow}B\overset{y}{\longrightarrow}C\overset{\delta}{\dashrightarrow}$$
We usually do not write this $``\delta"$ if it is not used in the argument.

\item[(3)] Let $A\overset{x}{\longrightarrow}B\overset{y}{\longrightarrow}C\overset{\delta}{\dashrightarrow}$ and $A^{\prime}\overset{x^{\prime}}{\longrightarrow}B^{\prime}\overset{y^{\prime}}{\longrightarrow}C^{\prime}\overset{\delta^{\prime}}{\dashrightarrow}$ be any pair of $\E$-triangles. If a triplet $(a,b,c)$ realizes $(a,c)\colon\delta\to\delta^{\prime}$, then we write it as
$$\xymatrix{
A \ar[r]^x \ar[d]^a & B\ar[r]^y \ar[d]^{b} & C\ar@{-->}[r]^{\del}\ar[d]^c&\\
A'\ar[r]^{x'} & B' \ar[r]^{y'} & C'\ar@{-->}[r]^{\del'} &}$$
and call $(a,b,c)$ a {\it morphism of $\E$-triangles}.

\item[{\rm (4)}] An object $P\in\mathcal{C}$ is called {\it projective} if
for any $\E$-triangle $A\overset{x}{\longrightarrow}B\overset{y}{\longrightarrow}C\overset{\delta}{\dashrightarrow}$ and any morphism $c\in\mathcal{C}(P,C)$, there exists $b\in\mathcal{C}(P,B)$ satisfying $yb=c$.
We denote the subcategory of projective objects by $\mathcal P\subseteq\mathcal{C}$. Dually, the subcategory of injective objects is denoted by $\mathcal I\subseteq\mathcal{C}$.

\item[{\rm (5)}] We say that $\mathcal C$ {\it has enough projective objects} if
for any object $C\in\mathcal C$, there exists an $\E$-triangle
$A\overset{x}{\longrightarrow}P\overset{y}{\longrightarrow}C\overset{\delta}{\dashrightarrow}$
satisfying $P\in\mathcal P$. Dually we can define $\mathcal C$ {\it has enough injective objects}.
\end{itemize}
\end{definition}
Let $\mathcal C$ be extriangulated category with enough projectives and enough injectives, and $\mathcal X$ a subcategory of $\mathcal C$.
We denote $\Omega\mathcal X={\rm CoCone}(\mathcal P,\mathcal X)$, that is to say, $\Omega\mathcal X$ is the subcategory of $\mathcal C$ consisting of
objects $\Omega X$ such that there exists an $\E$-triangle:
$$\Omega X\overset{a}{\longrightarrow}P\overset{b}{\longrightarrow}X\overset{}{\dashrightarrow},$$
with $P\in\P$ and $X\in\mathcal X$. We call $\Omega$ the \emph{syzygy} of $\mathcal X$.
Dually we define the \emph{cosyzygy} of $\mathcal X$ by $\Sigma\mathcal X={\rm Cone}(\mathcal X,\I)$.
Namely, $\Sigma\mathcal X$ is the subcategory of $\mathcal C$ consisting of
objects $\Sigma X$ such that there exists an $\E$-triangle:
$$X\overset{c}{\longrightarrow}I\overset{d}{\longrightarrow}\Sigma X\overset{}{\dashrightarrow},$$
with $I\in\I$ and $X\in\mathcal X$. For more details, see \cite[Definition 4.2 and Proposition 4.3]{LN}.

For a subcategory $\mathcal X\subseteq\mathcal C$, put $\Omega^0\mathcal X=\mathcal X$, and define $\Omega^k\mathcal X$ for $k>0$ inductively by
$$\Omega^k\mathcal X=\Omega(\Omega^{k-1}\mathcal X)={\rm CoCone}(\P,\Omega^{k-1}\mathcal X).$$
We call $\Omega^k\mathcal X$ the $k$-th syzygy of $\mathcal X$. Dually we define the $k$-th cosyzygy $\Sigma^k\mathcal X$ by
$\Sigma^0\X=\mathcal X$ and $\Sigma^k\mathcal X={\rm Cone}(\Sigma^{k-1}\mathcal X,\I)$ for $k>0$.

Liu and Nakaoka \cite{LN} defined higher extension groups in an extriangulated category with have
enough projectives and enough injectives as $\E(A,\Sigma^kB)\simeq\E(\Omega^kA,B)$ for $k\geq0$.
For convenience, we denote $\E(A,\Sigma^kB)\simeq\E(\Omega^kA,B)$ by $\E^{k+1}(A,B)$ for $k\geq0$.
They proved the following.
\begin{lemma}\label{022}
Let $\mathcal C$ be an extriangulated category with enough projectives and enough injectives. Assume that
$$\xymatrix{A\ar[r]^{f}& B\ar[r]^{g}&C\ar@{-->}[r]^{\delta}&}$$
is an $\E$-triangle in $\mathcal C$. Then for any object $X\in\mathcal C$ and $k\geq1$, we have the following exact sequences:
$$\cdots\xrightarrow{}\E^k(X,A)\xrightarrow{}\E^k(X,B)\xrightarrow{}\E^k(X,C)\xrightarrow{}\E^{k+1}(X,A)\xrightarrow{}\E^{k+1}(X,B)\xrightarrow{}\cdots;$$
$$\cdots\xrightarrow{}\E^k(C,X)\xrightarrow{}\E^k(B,X)\xrightarrow{}\E^k(A,X)\xrightarrow{}\E^{k+1}(C,X)\xrightarrow{}\E^{k+1}(B,X)\xrightarrow{}\cdots.$$
\end{lemma}
Next, we always assume that any extrianglated category satisfies the (WIC) condition,
see \cite[Condition 5.8]{NP}. We briefly recall the concepts and basic properties of recollements of extriangulated categories from \cite{WWZ}.
We omit some details here, but the reader can find them in \cite{WWZ}.

\begin{definition}\label{recollement}{\rm \cite[Definition 3.1]{WWZ}}
Let $\mathcal{A}$, $\mathcal{B}$ and $\mathcal{C}$ be three extriangulated categories. A \emph{recollement} of $\mathcal{B}$ relative to
$\mathcal{A}$ and $\mathcal{C}$, denoted by ($\mathcal{A}$, $\mathcal{B}$, $\mathcal{C}$), is a diagram
\begin{equation}\label{recolle}
  \xymatrix{\mathcal{A}\ar[rr]|{i_{*}}&&\ar@/_1pc/[ll]|{i^{*}}\ar@/^1pc/[ll]|{i^{!}}\mathcal{B}
\ar[rr]|{j^{\ast}}&&\ar@/_1pc/[ll]|{j_{!}}\ar@/^1pc/[ll]|{j_{\ast}}\mathcal{C}}
\end{equation}
given by two exact functors $i_{*},j^{\ast}$, two right exact functors $i^{\ast}$, $j_!$ and two left exact functors $i^{!}$, $j_\ast$, which satisfies the following conditions:
\begin{itemize}
  \item [(R1)] $(i^{*}, i_{\ast}, i^{!})$ and $(j_!, j^\ast, j_\ast)$ are adjoint triples.
  \item [(R2)] $\Im i_{\ast}=\Ker j^{\ast}$.
  \item [(R3)] $i_\ast$, $j_!$ and $j_\ast$ are fully faithful.
  \item [(R4)] For each $X\in\mathcal{B}$, there exists a left exact $\mathbb{E}_\mathcal{B}$-triangle sequence
  \begin{equation}\label{first}
  \xymatrix{i_\ast i^! X\ar[r]^-{\theta_X}&X\ar[r]^-{\vartheta_X}&j_\ast j^\ast X\ar[r]&i_\ast A}
   \end{equation}
  with $A\in \mathcal{A}$, where $\theta_X$ and  $\vartheta_X$ are given by the adjunction morphisms.
  \item [(R5)] For each $X\in\mathcal{B}$, there exists a right exact $\mathbb{E}_\mathcal{B}$-triangle sequence
  \begin{equation}\label{second}
  \xymatrix{i_\ast\ar[r] A' &j_! j^\ast X\ar[r]^-{\upsilon_X}&X\ar[r]^-{\nu_X}&i_\ast i^\ast X&}
   \end{equation}
 with $A'\in \mathcal{A}$, where $\upsilon_X$ and $\nu_X$ are given by the adjunction morphisms.
\end{itemize}
\end{definition}
We collect some properties of a recollement of extriangulated categories, which will be used in the sequel.
\begin{lemma}\label{CY}\rm{\cite[Proposition 3.3]{WWZ}} Let ($\mathcal{A}$, $\mathcal{B}$, $\mathcal{C}$) be a recollement of extriangulated categories as \rm{(\ref{recolle})}.

$(1)$ All the natural transformations
$$i^{\ast}i_{\ast}\Rightarrow\Id_\mathcal{A},~\Id_\mathcal{A}\Rightarrow i^{!}i_{\ast},~\Id_\mathcal{C}\Rightarrow j^{\ast}j_{!},~j^{\ast}j_{\ast}\Rightarrow\Id_\mathcal{C}$$
are natural isomorphisms.

$(2)$ $i^{\ast}j_!=0$ and $i^{!}j_\ast=0$.

$(3)$ $i^{\ast}$ preserves projective objects and $i^{!}$ preserves injective objects.

$(3')$ $j_{!}$ preserves projective objects and $j_{\ast}$ preserves injective objects.

$(4)$ If $i^{!}$ (resp. $j_{\ast}$) is  exact, then $i_{\ast}$ (resp. $j^{\ast}$) preserves projective objects.

$(4')$ If $i^{\ast}$ (resp. $j_{!}$) is  exact, then $i_{\ast}$ (resp. $j^{\ast}$) preserves injective objects.

$(5)$ If $\mathcal{B}$ has enough projectives, then $\mathcal{A}$ has enough projectives; if $\mathcal{B}$ has enough injectives, then $\mathcal{A}$ has enough injectives.

$(6)$  If $\mathcal{B}$ has enough projectives and $j_{\ast}$ is exact, then $\mathcal{C}$ has enough projectives ; if $\mathcal{B}$ has enough injectives and $j_{!}$ is exact, then $\mathcal{C}$ has enough injectives.

$(7)$  If $i^{\ast}$ is exact, then $j_{!}$ is  exact.

$(7')$ If $i^{!}$ is exact, then $j_{\ast}$ is exact.

\end{lemma}
Let us end this section with the following key lemma.
\begin{lemma}\label{impor}\rm
Let $(\mathcal{A},\mathcal{B},\mathcal{C})$ be a recollement of extriangulated categories and $n$ any positive integer.
\begin{enumerate}
\item[{\rm (1)}] If $\mathcal{B}$ has enough projectives and $i^{\ast}$ is exact, then $\mathbb{E}_{\mathcal{A}}(i^{\ast}X,Y)\cong\mathbb{E}_{\mathcal{B}}(X,i_{\ast}Y)$ for any $X\in\mathcal{B}$ and $Y\in\mathcal{A}$.

\item[{\rm (2)}] If $\mathcal{A}$ has enough projectives and $i^{!}$ is  exact, then $\mathbb{E}_{\mathcal{B}}(i_{\ast}X,Y)\cong\mathbb{E}_{\mathcal{A}}(X,i^{!}Y)$ for any $Y\in\mathcal{B}$ and $X\in\mathcal{A}$.

\item[{\rm (3)}] If $\mathcal{C}$ has enough projectives and $j_{!}$ is  exact, then $\mathbb{E}_{\mathcal{B}}(j_{!}X,Y)\cong\mathbb{E}_{\mathcal{C}}(X,j^{\ast}Y)$ for any $Y\in\mathcal{B}$ and $X\in\mathcal{C}$.

\item[{\rm (4)}] If $\mathcal{B}$ has enough projectives and $j_{\ast}$ is  exact, then $\mathbb{E}_{\mathcal{C}}(j^{\ast}X,Y)\cong\mathbb{E}_{\mathcal{B}}(X,j_{\ast}Y)$ for any $Y\in\mathcal{C}$ and $X\in\mathcal{B}$.
\end{enumerate}

\end{lemma}
\proof It is similar to \cite[Proposition 2.8]{MXZ}, we omit it.
\qed

\section{Glued $n$-cotorsion pairs}

\setcounter{equation}{0}

Let $\mathcal X$ be a class of objects in an extriangulated category $\mathcal C$.  For a nonnegative integer $m\geq 0$,
an $\mathcal X$-resolution of $C$ of length $m$ is a complex
$$X_m \to X_{m-1} \to \cdots \to X_1 \to X_0 \to C$$
where $X_k\in\mathcal X$ for any integer $0\leq k\leq m$. The above complex is determined by the following
$\E$-triangles:
$$\xymatrix{K_1\ar[r]&X_0\ar[r]&C\ar@{-->}[r]&}$$
$$\xymatrix{K_2\ar[r]&X_1\ar[r]&K_1\ar@{-->}[r]&}$$
$$\vdots$$
$$\xymatrix{K_{m-1}\ar[r]&X_{m-2}\ar[r]&K_{m-2}\ar@{-->}[r]&}$$
$$\xymatrix{X_m\ar[r]&X_{m-1}\ar[r]&K_{m-1}\ar@{-->}[r]&}$$
The \emph{resolution dimension of $C$ with respect to $\mathcal X$} (or the \emph{$\mathcal X$-resolution dimension} of $C$), denoted $\resdim_{\mathcal X}(C)$, is defined as the smallest nonnegative integer $m \geq 0$ such that $C$ has a $\mathcal X$-resolution of length $m$. If such $m$ does not exist, we set
$\resdim_{\mathcal X}(C) := \infty$.
Dually, we have the concepts of \emph{$\mathcal X$-coresolutions of $C$ of length $m$} and of  \emph{coresolution dimension of $C$ with respect to $\mathcal X$}, denoted by $\coresdim_{\mathcal X}(C)$.

Define
\begin{align*}
\mathcal X^\wedge_m & := \{ C \in \mathcal C~|~\resdim_{\mathcal X}(C) \leq m \}, \\
\mathcal X^\vee_m & := \{ C \in \mathcal C~|~\coresdim_{\mathcal X}(C) \leq m \}.
\end{align*}
In particular, $\mathcal X^\wedge_0=\mathcal X$ and $\mathcal X^\vee_0=\mathcal X$.
\medskip

We first recall the notion of left (resp. right) $n$-cotorsion pair from \cite{HZ}.
\begin{definition}\label{d1}{\cite[Definition 3.1]{HZ}}
Let $\mathcal C$ be an extriangulated category with enough projectives and enough injectives, and let $\mathcal X$ and $\mathcal Y$ be two classes of objects of $\mathcal C$. We call that $(\mathcal X,\mathcal Y)$ is a left $n$-cotorsion pair in $\mathcal C$ if the following conditions are satisfied:
\begin{enumerate}
\item[{\rm (1)}] $\mathcal X$ is closed under direct summands.

\item[{\rm (2)}] $\E^k(\mathcal X,\mathcal Y) = 0$~ for any $1\leq k \leq n$.

\item[{\rm (3)}] For any object $C \in \mathcal C$, there exists an $\E$-triangle
$$\xymatrix{K\ar[r]& X\ar[r]&C\ar@{-->}[r]&}$$
where $X\in \mathcal X$ and $K\in\mathcal Y^{\wedge}_{n-1}$.
\end{enumerate}
Dually, we can define a right $n$-cotorsion pair. If $(\mathcal X,\mathcal Y)$ is both a left and right $n$-cotorsion pair in $\mathcal C$, we call $(\mathcal X,\mathcal Y)$ an $n$-cotorsion pair in $\mathcal C$.
\end{definition}
Note that when $n=1$, an $n$-cotorsion pair is just
a cotorsion pair in the sense of Nakaoka-Palu, see \cite[Definition 4.1]{NP}.

Now we give some examples of $n$-cotorsion pairs. These examples come from \cite{HZ}.
\begin{example}
\begin{itemize}
\item [\rm(a)] Let $\mathcal C$ be an extriangulated category with enough projectives and enough injectives.
It is clear that  both $(\P,\mathcal C)$ and $(\mathcal C,\I)$ are $n$-cotorsion pair.

\item [\rm(b)] We denote by ``$\circ$" in the Auslander-Reiten quiver the indecomposable objects belong to a subcategory.
Let $\Lambda$ be the algebra given by the following quiver with relations\emph{:}
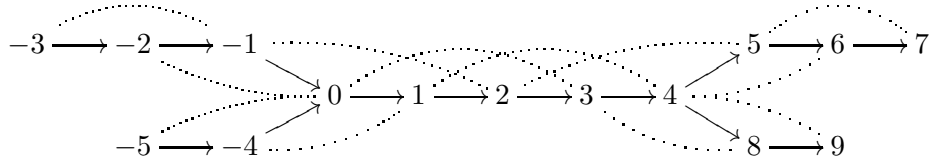
$$\xymatrix@C=0.7cm@R0.2cm{
&\\
-3 \ar[r] \ar@{.}@/^15pt/[rr] &-2 \ar@{.}@/_6pt/[drr]\ar[r] &-1 \ar@{.}@/^8pt/[drrr]\ar[dr] &&&&&&5 \ar[r] \ar@{.}@/^15pt/[rr] &6 \ar[r] &7\\
&&&0 \ar[r] \ar@{.}@/^18pt/[rrr] & 1 \ar[r] \ar@{.}@/^18pt/[rrr] &2 \ar[r] \ar@{.}@/^8pt/[urrr] & 3 \ar[r] \ar@{.}@/_8pt/[drr] &4 \ar[ur] \ar[dr] \ar@{.}@/_6pt/[urr] \ar@{.}@/^6pt/[drr] \\
&-5 \ar[r] \ar@{.}@/^6pt/[urr] &-4 \ar[ur] \ar@{.}@/_8pt/[urr] &&&&&&8 \ar[r] &9
}
$$
There exists a $3$-cluster tilting subcategory $\mathcal X$ of~~$\mathcal C={\rm mod}\Lambda$\emph{:}
$$\xymatrix@C=0.2cm@R0.2cm{
&&&\circ \ar[dr] &&&&\circ \ar[dr] &&\circ \ar[dr] &&\circ \ar[dr] &&\circ \ar[dr] &&\circ \ar[dr] &&&&\circ \ar[dr]\\
&\X: &\circ \ar@{.}[rr] \ar[ur] &&\cdot \ar[dr] \ar@{.}[rr] &&\cdot \ar[dr] \ar[ur] \ar@{.}[rr] &&\cdot \ar[dr] \ar[ur] \ar@{.}[rr] &&\cdot \ar[dr] \ar[ur] \ar@{.}[rr] &&\cdot \ar[dr] \ar[ur] \ar@{.}[rr] &&\cdot \ar[dr] \ar[ur] \ar@{.}[rr] &&\cdot \ar[dr] \ar@{.}[rr] &&\cdot  \ar[ur] \ar@{.}[rr] &&\circ\\
&&&&&\circ \ar[dr] \ar[ur] \ar@{.}[rr] &&\cdot \ar[ur] \ar@{.}[rr] &&\circ \ar[ur] \ar@{.}[rr] &&\cdot \ar[ur] \ar@{.}[rr] &&\circ \ar[ur] \ar@{.}[rr] &&\cdot \ar[dr] \ar[ur] \ar@{.}[rr] &&\circ \ar[dr] \ar[ur]\\
\circ \ar@{.}[rr] \ar[dr] &&\cdot \ar@{.}[rr] \ar[dr] &&\cdot \ar[ur] \ar@{.}[rr] &&\circ &&&&&&&&&&\circ \ar[ur] \ar@{.}[rr] &&\cdot \ar[dr] \ar@{.}[rr] &&\cdot \ar[dr] \ar@{.}[rr] &&\circ\\
&\circ \ar[ur] &&\circ \ar[ur] &&&&&&&&&&&&&&&&\circ \ar[ur] && \circ \ar[ur]
}
$$
By  \cite[Theorem 3.5]{HZ}, we have that $(\mathcal X,\mathcal X)$ is a $2$-cotorsion pair in $\mathcal C$.

\item [\rm(c)]Let $\Lambda$ be a finite-dimensional algebra of global dimension at most $n$.
We denote the Serre functor of $D^b({\rm mod}\Lambda)$
by $\mathbb{S}$, where $D^b({\rm mod}\Lambda)$ is the bounded derived category of ${\rm mod}\Lambda$.
If $\Lambda$ is $n$-representation finite, that is to say, the
module category ${\rm mod}\Lambda$ has an $n$-cluster tilting object, by \cite[Theorem 1.23]{I},
we obtain that the subcategory $$\mathcal X:={\rm add}\{\mathbb{S}^k\Lambda[-nk]~|~k\in\mathbb{Z}\}$$
of $D^b({\rm mod}\Lambda)$ is $n$-cluster tilting. By \cite[Theorem 3.5]{HZ},
 we have that $(\mathcal X,\mathcal X)$ is an $(n-1)$-cotorsion pair in $\mathcal C$.
\end{itemize}
\end{example}

\begin{definition}\label{her}
A left or right $n$-cotorsion pair $(\mathcal X,\mathcal Y)$ in $\mathcal C$ is hereditary if $\E^{n+1}(\mathcal X,\mathcal Y) = 0$.
\end{definition}

The following useful result can be found in \cite{WWZ}.

\begin{lemma}\label{1}\rm{\cite[Proposition 3.4]{WWZ}}
Let ($\mathcal{A}$, $\mathcal{B}$, $\mathcal{C}$) be a recollement of extriangulated categories and $X\in\mathcal{B}$. Then the following statements hold.

$(1)$ If $i^{!}$ is exact, there exists an $\mathbb{E}_\mathcal{B}$-triangle
  \begin{equation*}\label{third}
  \xymatrix{i_\ast i^! X\ar[r]^-{\theta_X}&X\ar[r]^-{\vartheta_X}&j_\ast j^\ast X\ar@{-->}[r]&}
   \end{equation*}
 where $\theta_X$ and  $\vartheta_X$ are given by the adjunction morphisms.

$(2)$ If $i^{\ast}$ is exact, there exists an $\mathbb{E}_\mathcal{B}$-triangle
  \begin{equation*}\label{four}
  \xymatrix{ j_! j^\ast X\ar[r]^-{\upsilon_X}&X\ar[r]^-{\nu_X}&i_\ast i^\ast X \ar@{-->}[r]&}
   \end{equation*}
where $\upsilon_X$ and $\nu_X$ are given by the adjunction morphisms.
\end{lemma}

Our first main result is the following.

\begin{theorem}\label{main}
 Let $\mathcal A,\mathcal B,\mathcal C$ be three extriangulated categories with enough projectives and enough injectives. Assume that
 $\mathcal B$ admits a recollement relative to $\mathcal A$ and $\mathcal C$ as follows
$$\xymatrix{\mathcal{A}\ar[rr]|{i_{*}}&&\ar@/_1pc/[ll]|{i^{*}}\ar@/^1pc/[ll]|{i^{!}}\mathcal{B}
\ar[rr]|{j^{\ast}}&&\ar@/_1pc/[ll]|{j_{!}}\ar@/^1pc/[ll]|{j_{\ast}}\mathcal{C}}$$
 If $(\mathcal{T}_{1},\mathcal{F}_{1})$ and $(\mathcal{T}_{2},\mathcal{F}_{2})$ are $n$-cotorsion pairs in $\mathcal{A}$ and $\mathcal{C}$, respectively. Define
$$\mathcal{T}:=\{B\in \mathcal{B}~|~i^{\ast }B\in\mathcal{T}_{1}~\text{and}~j^{\ast}B\in \mathcal{T}_{2}  \}
~~\mbox{and}~~
\mathcal{F}=:\{B\in \mathcal{B}~|~i^{!}B\in\mathcal{F}_{1}~\text{and}~j^{\ast}B\in \mathcal{F}_{2}  \}.$$
Then the following statements hold, if both $i^{!}$ and $i{^{\ast}}$ are exact.
\begin{itemize}

\item[\rm (1)]$(\mathcal{T}_{1}, \mathcal{F}_{1})=(i^{\ast }\mathcal{T}, i^{!}\mathcal{F})$ and $(\mathcal{T}_{2}, \mathcal{F}_{2})=(j^{\ast }\mathcal{T}, j^{\ast }\mathcal{F})$;

\item[\rm (2)] $\mathcal{T}^\vee_{n-1}=\{B\in \mathcal{B}~|~i^{\ast }B\in{\mathcal{T}_1}^\vee_{n-1}~\text{and}~j^{\ast}B\in {\mathcal{T}_2}^\vee_{n-1}  \}$,

    $\mathcal{F}^\wedge_{n-1}=\{B\in \mathcal{B}~|~i^{!}B\in{\mathcal{F}_1}^\wedge_{n-1}~\text{and}~j^{\ast}B\in {\mathcal{F}_2}^\wedge_{n-1}  \}$;

\item[\rm (3)] \label{3} $(\mathcal{T},\mathcal{F})$ is an $n$-cotorsion pair in $\mathcal{B}$.
 We call $(\mathcal{T},\mathcal{F})$  is ``glued" by $(\mathcal T_1,\mathcal F_1)$
and $(\mathcal T_2,\mathcal F_2)$.

\item[\rm (4)] If $(\mathcal{T}_{1},\mathcal{F}_{1})$ and $(\mathcal{T}_{2},\mathcal{F}_{2})$ are hereditary $n$-cotorsion pairs in $\mathcal{A}$ and $\mathcal{C}$, respectively, then $(\mathcal{T},\mathcal{F})$ is a hereditary $n$-cotorsion pair in $\mathcal{B}$.

\end{itemize}
\end{theorem}

\proof
(1) It is obvious that $i^{\ast}\mathcal{T}\subseteq\mathcal{T}_{1}.$ On the other hand, for any ${T}_{1}\in\mathcal{T}_{1},$ since $$i^{\ast}i_{\ast}{T}_{1}\cong{T}_{1}\in\mathcal{T}_{1},~~~ j^{\ast}i_{\ast}{T}_{1}=0\in\mathcal{T}_{2},$$
we have $\mathcal {T}_{1}\subseteq i^{\ast}\mathcal{T}.$
Similarly, we obtain $\mathcal{F}_{1}=i^{!}\mathcal{F},$   $\mathcal{T}_{2}=j^{\ast }\mathcal{T},  \mathcal{F}_{2}=j^{\ast }\mathcal{F}.$

(2) For any $B\in \mathcal{T}^\vee_{n-1}$, we have a complex as follows
$$\xymatrix@C=0.3cm{
 B \ar[rr]^{}&& X_0\ar[rr]^{}\ar[dr]_{}&& X_1 \ar[rr]^{}\ar[dr]_{}&&\cdots\ar[rr]^{}\ar[dr]_{} &&X_{n-3 }\ar[dr]_{}\ar[rr]^{}&& X_{n-2 }\ar[rr]^{} &&X_{n-1 } \\
                &&&K_1\ar[ru]^{} && K_2 \ar[ru]^{} &&K_{n-3 }\ar[ru]^{}&&K_{n-2 }\ar[ru]^{} &&}
$$
where $$\xymatrix{B\ar[r]^{}& X_0\ar[r]^{}&K_1\ar@{-->}[r]^{}&},$$ $$\xymatrix{K_i\ar[r]^{}& X_i\ar[r]^{}&K_{i+1}\ar@{-->}[r]^{}&},~~~1\leq i\leq n-3$$
$$\xymatrix{K_{n-2 }\ar[r]^{}& X_{n-2 }\ar[r]^{}&X_{n-1 }\ar@{-->}[r]^{}&}$$ are $\E_\mathcal B$-triangles, $X_0,X_1,\cdots,X_{n-1 }\in \mathcal{T}$. Since $j^{\ast }$ is exact, we obtain a complex
$$\xymatrix@C=0.3cm{
 j^{\ast }B \ar[rr]^{}&& j^{\ast }X_0\ar[rr]^{}\ar[dr]_{}&& j^{\ast }X_1 \ar[rr]^{}\ar[dr]_{}&&\cdots\ar[rr]^{}\ar[dr]_{} && j^{\ast }X_{n-2 }\ar[rr]^{} &&j^{\ast }X_{n-1 } \\
                &&&j^{\ast }K_1\ar[ru]^{} && j^{\ast }K_2 \ar[ru]^{} &&j^{\ast }K_{n-2 }\ar[ru]^{}&&}
$$
Note that $j^{\ast }X_i\in j^{\ast }\mathcal{T}=\mathcal{T}_{2}$, $i=0,1,\cdots,n-1$, so we have $j^{\ast}B\in {\mathcal{T}_2}^\vee_{n-1}$. Since $i^{\ast }$ is exact, there is a complex
$$\xymatrix@C=0.3cm{
 i^{\ast }B \ar[rr]^{}&& i^{\ast }X_0\ar[rr]^{}\ar[dr]_{}&& i^{\ast }X_1 \ar[rr]^{}\ar[dr]_{}&&\cdots\ar[rr]^{}\ar[dr]_{} && i^{\ast }X_{n-2 }\ar[rr]^{} &&i^{\ast }X_{n-1 } \\
                &&&i^{\ast }K_1\ar[ru]^{} && i^{\ast }K_2 \ar[ru]^{} &&i^{\ast }K_{n-2 }\ar[ru]^{}&&}
$$
Note that $i^{\ast }X_i\in i^{\ast }\mathcal{T}=\mathcal{T}_{1}$, $i=0,1,\cdots,n-1$, so we have $i^{\ast}B\in {\mathcal{T}_1}^\vee_{n-1}$.

Conversely, let $B\in\mathcal{B}$ satisfy $i^{\ast}B\in {\mathcal{T}_1}^\vee_{n-1}$ and $j^{\ast}B\in {\mathcal{T}_2}^\vee_{n-1}$. Then there exist the following two complexes
\begin{equation}\xymatrix@C=0.3cm{
 i^{\ast }B \ar[rr]^{}&& Y_0\ar[rr]^{}\ar[dr]_{}&& Y_1 \ar[rr]^{}\ar[dr]_{}&&\cdots\ar[rr]^{}\ar[dr]_{} && Y_{n-2 }\ar[rr]^{} &&Y_{n-1 } \\
                &&&K'_1\ar[ru]^{} && K'_2 \ar[ru]^{} &&K'_{n-2 }\ar[ru]^{}&&}
\end{equation}
and
\begin{equation}\xymatrix@C=0.3cm{
j^{\ast }B \ar[rr]^{}&& Z_0\ar[rr]^{}\ar[dr]_{}&& Z_1 \ar[rr]^{}\ar[dr]_{}&&\cdots\ar[rr]^{}\ar[dr]_{} && Z_{n-2 }\ar[rr]^{} &&Z_{n-1 } \\
                &&&K''_1\ar[ru]^{} &&K''_2 \ar[ru]^{} &&K''_{n-2 }\ar[ru]^{}&&}
\end{equation}
where $Y_i\in \mathcal{T}_1$ and $Z_i\in \mathcal{T}_2$, $i=0,1,\cdots,n-1$. Since $i_{\ast }$ and  $j_{! }$ are exact, we have the following two complexes
\begin{equation}\xymatrix@C=0.3cm{
 i_{\ast }i^{\ast }B \ar[rr]^{}&& i_{\ast }Y_0\ar[rr]^{}\ar[dr]_{}&& i_{\ast }Y_1 \ar[rr]^{}\ar[dr]_{}&&\cdots\ar[rr]^{}\ar[dr]_{} && i_{\ast }Y_{n-2 }\ar[rr]^{} &&i_{\ast }Y_{n-1 } \\
                &&&i_{\ast }K'_1\ar[ru]^{} &&i_{\ast } K'_2 \ar[ru]^{} &&i_{\ast }K'_{n-2 }\ar[ru]^{}&&}
\end{equation}
and
\begin{equation}\xymatrix@C=0.3cm{
j_{! }j^{\ast }B \ar[rr]^{}&& j_{! }Z_0\ar[rr]^{}\ar[dr]_{}&& j_{! }Z_1 \ar[rr]^{}\ar[dr]_{}&&\cdots\ar[rr]^{}\ar[dr]_{} && j_{! }Z_{n-2 }\ar[rr]^{} &&j_{! }Z_{n-1 } \\
                &&&j_{! }K''_1\ar[ru]^{} &&j_{! }K''_2 \ar[ru]^{} &&j_{! }K''_{n-2 }\ar[ru]^{}&&}
\end{equation}
where $i_{\ast }Y_i\in \mathcal{T}$ and $j_{! }Z_i\in \mathcal{T}$, $i=0,1,\cdots,n-1$. Since $i^{\ast }$ is exact, there exists an $\mathbb{E}_\mathcal{B}$-triangle by Lemma \ref{1}
  $$
  \xymatrix{ j_! j^\ast B\ar[r]^-{}&B\ar[r]^-{}&i_\ast i^\ast B \ar@{-->}[r]&}
  $$
Since $i^\ast$ and $i^!$ are exact, by Lemma \ref{impor}, we have $\mathbb{E}_{\mathcal{B}}(i_\ast i^\ast B, j_{! }Z_0)\cong\mathbb{E}_{\mathcal{A}}(i^\ast B, i^{! }j_{! }Z_0)=0$. Thus there is a commutative diagram of $\E_\mathcal B$-triangles
$$\xymatrix{j_! j^\ast B\ar[r]^{}\ar[d]^{}&B\ar[r]^{}\ar[d]^{}&i_\ast i^\ast B\ar[d]^{}\ar@{-->}[r]&\\
j_{! }Z_0\ar[d]^{}\ar[r]^{}&j_{! }Z_0\oplus i_{\ast }Y_0\ar[d]^{}\ar[r]^{}&i_{\ast }Y_0\ar[d]^{}\ar@{-->}[r]&\\
j_{! }K''_1\ar[r]^{}\ar@{-->}[d]&K_1\ar@{-->}[d]\ar[r]^{}&i_{\ast }K'_1\ar@{-->}[d]\ar@{-->}[r]&\\
&&&}$$
Repeat the process for $$
  \xymatrix{ j_{! }K''\ar[r]^-{}&K_1\ar[r]^-{}&i_{\ast }K'_1\ar@{-->}[r]&},
  $$
we end up with $B\in \mathcal{T}^\vee_{n-1}$. One can similarly prove that  $$\mathcal{F}^\wedge_{n-1}=\{B\in \mathcal{B}~|~i^{!}B\in{\mathcal{F}_1}^\wedge_{n-1}~\text{and}~j^{\ast}B\in {\mathcal{F}_2}^\wedge_{n-1}  \}$$.

(3) We claim that $T\in\mathcal{T}$ if and only if $\E_\mathcal{B}^k(T,\mathcal{F}) = 0$, $1\leq k\leq n$. Since $i^{\ast }$ is exact, there exists an $\mathbb{E}_\mathcal{B}$-triangle by Lemma \ref{1}
  $$
  \xymatrix{ j_! j^\ast T\ar[r]^-{}&T\ar[r]^-{}&i_\ast i^\ast T \ar@{-->}[r]&.}
  $$
So we obtain the following exact sequence  Lemma \ref{022}
$$
  \xymatrix{ \E_\mathcal{B}^k(i_\ast i^\ast T, F)\ar[r]^-{}&\E_\mathcal{B}^k( T, F)\ar[r]^-{}&\E_\mathcal{B}^k(j_!j^\ast T, F)&}
  $$
for any $F\in\mathcal{F}$. Note that $i^!$ and  $j_!$ are exact, by Lemma \ref{impor}, we have $ \E_\mathcal{B}^k(i_\ast i^\ast T, F)\cong \E_\mathcal{A}^k( i^\ast T, i^! F)=0$ and $ \E_\mathcal{B}^k(j_!j^\ast T, F)\cong \E_\mathcal{C}^k( j^\ast T, j^\ast F)=0$, $1\leq k\leq n$. So we get $\E_\mathcal{B}^k( T, F)=0$ for $1\leq k\leq n$. Conversely, suppose
$\E_\mathcal{B}^k(T,\mathcal{F}) = 0$, $1\leq k\leq n$. For any $F_1\in\mathcal{F}_1$ and $F_2\in\mathcal{F}_2$, since $i_\ast F_1\in\mathcal{F}$ and $j_\ast F_2\in\mathcal{F}$, by Lemma \ref{impor}, we have $\E_\mathcal{A}^k(i^\ast T, F_1)\cong \E_\mathcal{B}^k(  T, i_\ast F_1)=0$ and $\E_\mathcal{C}^k(j^\ast T, F_2)\cong \E_\mathcal{B}^k(  T, j_\ast F_2)=0$. This shows $i^\ast T\in T_1$ and $j^\ast T\in T_2$. That is to say $T\in \mathcal{T}$. Using a similar argument, one can prove that $F\in\mathcal{F}$ if and only if $\E_\mathcal{B}^k(\mathcal{T},F) = 0$, $1\leq k\leq n$.

For any $B\in\mathcal B$, since $(\mathcal{T}_{2},\mathcal{F}_{2})$ is an $n$-cotorsion pair in $\mathcal{C}$, there exists an $\mathbb{E}_\mathcal{C}$-triangle
 $$
  \xymatrix{K\ar[r]^-{}&A\ar[r]^-{}&j^\ast B \ar@{-->}[r]&},
  $$
where $A\in\mathcal{T}_{2}$ and $K\in{\mathcal{F}_2}^\wedge_{n-1} $. Since $j_\ast$ is exact, we obtain an $\mathbb{E}_\mathcal{B}$-triangle
 $$
  \xymatrix{j_\ast K\ar[r]^-{}&j_\ast A\ar[r]^-{}&j_\ast j^\ast B \ar@{-->}[r]&}.
  $$
Let $\vartheta_B: B\rightarrow j_\ast j^\ast B$ be the adjunction morphism, by the dual of \cite[Proposition 1.20]{LN}, we have the following commutative diagram
\begin{equation}\label{t0}
\begin{array}{l}
\xymatrix{
j_\ast K \ar[r]\ar@{=}[d]& M\ar[r]^{} \ar[d]^{} &B\ar@{-->}[r]^{}\ar[d]^{\vartheta_B}&\\
j_\ast K\ar[r]^{} & j_\ast A \ar[r]^{} & j_\ast j^\ast B\ar@{-->}[r]^{} &}
\end{array}
\end{equation}
of $\mathbb{E}_\mathcal{B}$-triangles. This show that $j^\ast M\cong j^\ast j_\ast A\cong A\in\mathcal{T}_{2}$. Since $i^\ast M\in\mathcal{A}$ and $(\mathcal{T}_{1},\mathcal{F}_{1})$ is an $n$-cotorsion pair in $\mathcal{A}$, there exists an $\mathbb{E}_\mathcal{A}$-triangle
 $$
  \xymatrix{L\ar[r]^-{}&C\ar[r]^-{}&i^\ast M \ar@{-->}[r]&},
  $$
where $C\in\mathcal{T}_{1}$ and $L\in{\mathcal{F}_1}^\wedge_{n-1} $. We have an $\mathbb{E}_\mathcal{B}$-triangle since $i_\ast$ is exact
 $$
  \xymatrix{i_\ast L\ar[r]^-{}&i_\ast C\ar[r]^-{}&i_\ast i^\ast M \ar@{-->}[r]&}.
  $$
Let $\vartheta_M: M\rightarrow i_\ast i^\ast M$ be the adjunction morphism, by the dual of \cite[Proposition 1.20]{LN},  we have the following commutative diagram
\begin{equation}\label{t1}
\begin{array}{l}\xymatrix{
i_\ast L\ar[r]\ar@{=}[d]& N\ar[r]^{} \ar[d]^{} &M\ar@{-->}[r]^{}\ar[d]^{\vartheta_M}&\\
i_\ast L\ar[r]^{} & i_\ast C \ar[r]^{} & i_\ast i^\ast M\ar@{-->}[r]^{} &}
\end{array}
\end{equation}
of $\mathbb{E}_\mathcal{B}$-triangles. This show that $i^\ast N\cong i^\ast i_\ast C\cong C\in\mathcal{T}_{1}$. Applying $j^\ast$ to the first row of (\ref{t1}), we have $j^\ast N\cong j^\ast M\cong A\in \mathcal{T}_{2}$. So $N\in\mathcal{T}$. Considering the first rows of (\ref{t0}) and (\ref{t1}), by $\rm (ET4)^{op}$, we have a commutative diagram of  $\E_\mathcal{B}$-triangles
$$\xymatrix{i_\ast L\ar[r]^{}\ar@{=}[d]&W\ar[r]^{}\ar[d]^{}&j_\ast K\ar[d]\ar@{-->}[r]&\\
i_\ast L\ar[r]^{}&N\ar[d]^{}\ar[r]^{}&M\ar[d]^{}\ar@{-->}[r]&\\
&B\ar@{-->}[d]\ar@{=}[r]&B\ar@{-->}[d]\\
&&&}$$
We claim that the $\E_\mathcal{B}$-triangle $
  \xymatrix{W\ar[r]^-{}&N \ar[r]^-{}&B\ar@{-->}[r]&}
  $ is what we want. To do that, we just need to prove $W\in{\mathcal{F}}^\wedge_{n-1} $. Applying $j^\ast$ to the $\mathbb{E}_\mathcal{B}$-triangle
$$
  \xymatrix{i_\ast L\ar[r]^-{}&W \ar[r]^-{}&j_\ast K\ar@{-->}[r]&,}
  $$
we have $j^\ast W\cong j^\ast j_\ast K \cong K\in{\mathcal{F}_2}^\wedge_{n-1} $. Applying $i^!$ to the $\mathbb{E}_\mathcal{B}$-triangle
$$
  \xymatrix{i_\ast L\ar[r]^-{}&W \ar[r]^-{}&j_\ast K\ar@{-->}[r]&,}
  $$
we have $i^! W\cong i^! i_\ast L \cong L\in{\mathcal{F}_1}^\wedge_{n-1} $. So $W\in{\mathcal{F}}^\wedge_{n-1} $ by (2).

Using a similar argument, one can prove that for any object $B\in \mathcal{B}$, there exists an $\mathbb{E}_\mathcal{B}$-triangle
$$\xymatrix{B\ar[r]& D \ar[r]&Z\ar@{-->}[r]&},$$
where $D\in \mathcal{F}$ and $Z\in\mathcal{T}^{\vee}_{n-1}$.

(4) For any $T\in \mathcal{T}$, since $i^{\ast }$ is exact, there exists an $\mathbb{E}_\mathcal{B}$-triangle
  $$
  \xymatrix{ j_! j^\ast T\ar[r]^-{}&T\ar[r]^-{}&i_\ast i^\ast T \ar@{-->}[r]&}.
  $$
We obtain the following exact sequence Lemma \ref{022}
$$
  \xymatrix{ \E_\mathcal{B}^{n+1}(i_\ast i^\ast T, F)\ar[r]^-{}&\E_\mathcal{B}^{n+1}( T, F)\ar[r]^-{}&\E_\mathcal{B}^{n+1}(j_!j^\ast T, F),&}
  $$
where $F\in \mathcal F $. Since $i_\ast, i^!$ and $j_!, i^\ast$ are adjoint pairs, by Lemma \ref{impor}, we have $ \E_\mathcal{B}^{n+1}(i_\ast i^\ast T, F)\cong \E_\mathcal{A}^{n+1}( i^\ast T, i^! F)=0$ and $ \E_\mathcal{B}^{n+1}(j_!j^\ast T, F)\cong \E_\mathcal{C}^{n+1}( j^\ast T, j^\ast F)=0$. Therefore $\E_\mathcal{B}^{n+1}( T, F)=0$.
$(\mathcal{T},\mathcal{F})$ is an $n$-cotorsion pair in $\mathcal{B}$.
\qed
\begin{corollary}\rm
Assume that $(\mathcal{A},\mathcal{B},\mathcal{C})$ is a recollement of extriangulated categories with enough projectives and enough injectives. Let $(\mathcal{T}_{1},\mathcal{F}_{1})$ and $(\mathcal{T}_{2},\mathcal{F}_{2})$ are left (resp. right) $n$-cotorsion pairs in $\mathcal{A}$ and $\mathcal{C}$, respectively. Define
$$\mathcal{T}:=\{B\in \mathcal{B}~|~i^{\ast }B\in\mathcal{T}_{1}~\text{and}~j^{\ast}B\in \mathcal{T}_{2}  \}
~~\mbox{and}~~
\mathcal{F}=:\{B\in \mathcal{B}~|~i^{!}B\in\mathcal{F}_{1}~\text{and}~j^{\ast}B\in \mathcal{F}_{2}  \}.$$
Then $(\mathcal{T},\mathcal{F})$ is a left (resp. right) $n$-cotorsion pair in $\mathcal{B}$.

\end{corollary}

\begin{corollary}\rm
Assume that $(\mathcal{A},\mathcal{B},\mathcal{C})$ is a recollement of extriangulated categories with enough projectives and enough injectives.  Let $\mathcal X$ and $\mathcal Y$ are $(n+1)$-cluster tilting subcategories of $\mathcal{A}$ and $\mathcal{C}$, respectively. Define
$$\mathcal{T}:=\{B\in \mathcal{B}~|~i^{\ast }B\in\mathcal X~\text{and}~j^{\ast}B\in \mathcal Y  \}
~~\mbox{and}~~
\mathcal{F}=:\{B\in \mathcal{B}~|~i^{!}B\in\mathcal X~\text{and}~j^{\ast}B\in \mathcal Y \}.$$
Then $(\mathcal{T},\mathcal{F})$ is an $n$-cotorsion pair in $\mathcal{B}$.
\end{corollary}
\proof
This follows from Theorem \ref{main} and  \cite[Theorem 3.5]{HZ}.
\qed

By applying Theorem \ref{main} to abelian categories, and using the fact that any abelian
category can be viewed as an extriangulated category, we get the following result.

\begin{corollary}
{\rm In Theorem {\ref{main}}, when $(\mathcal{A},\mathcal{B},\mathcal{C})$ is a recollement of abelian categories, it is just Theorem {\rm 3.1} in {\rm \cite{CWW}}}.
\end{corollary}

Our second main result shows that the converse of Theorem \ref{main} holds true under certain conditions.
\begin{theorem}\label{main1}
 Let $\mathcal A,\mathcal B,\mathcal C$ be three extriangulated categories with enough projectives and enough injectives. Assume that
 $\mathcal B$ admits a recollement relative to $\mathcal A$ and $\mathcal C$ as follows
$$\xymatrix{\mathcal{A}\ar[rr]|{i_{*}}&&\ar@/_1pc/[ll]|{i^{*}}\ar@/^1pc/[ll]|{i^{!}}\mathcal{B}
\ar[rr]|{j^{\ast}}&&\ar@/_1pc/[ll]|{j_{!}}\ar@/^1pc/[ll]|{j_{\ast}}\mathcal{C}}$$
Assume that $(\mathcal{T},\mathcal{F})$ is an $n$-cotorsion pair in $\mathcal{B}$. If $i^{\ast }$ and $i^{!}$ are exact, $j_{\ast }j^{\ast }\mathcal{F}\subseteq \mathcal{F}$, $i_{\ast }i^{\ast }\mathcal{T}\subseteq \mathcal{T}$ and $i_{\ast }i^{\ast }\mathcal{F}\subseteq \mathcal{F}$, then
\begin{itemize}

\item[\rm (1)]$(i^{\ast }\mathcal{T}, i^{!}\mathcal{F})$ is an $n$-cotorsion pair in $\mathcal{A}$;

\item[\rm (2)] $(j^{\ast }\mathcal{T}, j^{\ast}\mathcal{F})$ is an $n$-cotorsion pair in $\mathcal{C}$.

\item[\rm (3)] If $(\mathcal{T},\mathcal{F})$ is a hereditary $n$-cotorsion pair in $\mathcal{B}$, then $(i^{\ast }\mathcal{T}, i^{!}\mathcal{F})$ and $(j^{\ast }\mathcal{T}, j^{\ast}\mathcal{F})$ are hereditary $n$-cotorsion pairs in $\mathcal{A}$ and $\mathcal{C}$, respectively.

\end{itemize}
\end{theorem}
\proof (1) We claim that $i_{\ast }i^{\ast }\mathcal{T}\subseteq \mathcal{T}$ if and only if $i_{\ast }i^{! }\mathcal{F}\subseteq \mathcal{F}$, $j_{\ast }j^{\ast }\mathcal{F}\subseteq \mathcal{F}$ if and only if $j_{!}j^{\ast }\mathcal{T}\subseteq \mathcal{T}$ and $i^{! }\mathcal{F}=i^{\ast}\mathcal{F}$. Indeed, since $i^{\ast }, j_{\ast }$ and $j_{!}$ are exact, by Lemma \ref{impor}, we have
$$\E_\mathcal{B}^k(T, i_{\ast }i^{! }F)\cong \E_\mathcal{A}^k(  i^{\ast }T, i^{!} F)\cong \E_\mathcal{B}^k(  i_{\ast }i^{\ast }T, F)$$ and
$$\E_\mathcal{B}^k(T, j_{\ast }j^{\ast}F)\cong \E_\mathcal{C}^k(  j^{\ast }T, j^{\ast} F)\cong \E_\mathcal{B}^k(  j_{!}j^{\ast }T, F),$$
where $T\in\mathcal{T}, F\in\mathcal{F}$, $1\leq k \leq n$. Note that $i^{\ast }i_{\ast }\cong \id_\mathcal{A}$ and $i_{\ast }i^{! }\mathcal{F}\subseteq \mathcal{F}$, we have $i^{! }\mathcal{F}\subseteq i^{\ast }\mathcal{F}$. Since $i^{! }i_{\ast }\cong \id_\mathcal{A}$ and $i_{\ast }i^{\ast}\mathcal{F}\subseteq \mathcal{F}$, we have $ i^{\ast }\mathcal{F}\subseteq i^{! }\mathcal{F}$. This shows $i^{! }\mathcal{F}=i^{\ast}\mathcal{F}$.

Next, we show that $T_1\in i^{\ast }\mathcal{T}$ if and only if $\E_\mathcal{A}^k(T_1, i^{!} \mathcal{F})=0$. Let $T_1\in i^{\ast }\mathcal{T}$, there is an object $T'\in \mathcal{T}$ such that $T_1=i^{\ast }T'$. Then for any $i^{! }F_1\in i^{! }\mathcal{F} $, since $i_{\ast }i^{!} F\subseteq \mathcal{F}$, by Lemma \ref{impor}, we have
$$\E_\mathcal{A}^k(T_1, i^{! }F_1)\cong \E_\mathcal{A}^k( i^{\ast }T', i^{! }F_1)\cong \E_\mathcal{B}^k(T', i_{\ast }i^{! }F_1)=0,$$
where $1\leq k\leq n$. On the other hand, if $\E_\mathcal{A}^k(T_1, i^{!} \mathcal{F})=0$, then for any $ i^{!}F_1\in i^{!} \mathcal{F}$, by Lemma \ref{impor}, we have $0=\E_\mathcal{A}^k(T_1, i^{!}F_1)\cong\E_\mathcal{B}^k(i_{\ast }T_1, F_1)$. So $i_{\ast }T_1\in\mathcal{T}$ and then $T_1\cong i^{\ast }i_{\ast }T_1\in i^{\ast }\mathcal{T}$.

Similarly, one can prove that $F_1\in i^{!}\mathcal{F}$ if and only if $\E_\mathcal{A}^k(i^{\ast} \mathcal{T}, F_1)=0, 1\leq k\leq n$.

For any $A\in\mathcal{A}$, since $(\mathcal{T},\mathcal{F})$ is an $n$-cotorsion pair in $\mathcal{B}$, there is an $\mathbb{E}_\mathcal{B}$-triangle
  $$
  \xymatrix{ K\ar[r]^-{}&B\ar[r]^-{}&i_\ast A\ar@{-->}[r]&},
  $$
where $B\in \mathcal{T}$ and $K\in{\mathcal{F}}^\wedge_{n-1}$. Since $i^{\ast }$ is exact, we obtain an $\mathbb{E}_\mathcal{A}$-triangle
\begin{equation}\label{want}
\begin{array}{l}
  \xymatrix{ i^{\ast }K\ar[r]^-{}&i^{\ast }B\ar[r]^-{}&A\ar@{-->}[r]&}.
\end{array}
\end{equation}
Note that  $K\in{\mathcal{F}}^\wedge_{n-1}$, there is a complex
$$\xymatrix@C=0.3cm{
 F_{n-1}\ar[rr]^{}&& F_{n-2}\ar[rr]^{}\ar[dr]_{}&&F_{n-3} \ar[rr]^{}\ar[dr]_{}&&\cdots\ar[rr]^{}\ar[dr]_{} &&F_{1}\ar[dr]_{}\ar[rr]^{}&& F_{0 }\ar[rr]^{} &&K \\
                &&&K_{n-2}\ar[ru]^{} && K_{n-1} \ar[ru]^{} &&K_{2 }\ar[ru]^{}&&K_{1 }\ar[ru]^{} &&}
$$
where $F_{k}\in \mathcal{F}, 1\leq k\leq n$. And then we have a complex as follows
$$\xymatrix@C=0.3cm{
 i^{\ast }F_{n-1}\ar[rr]^{}&& i^{\ast }F_{n-2}\ar[rr]^{}\ar[dr]_{}&&\cdots\ar[rr]^{}\ar[dr]_{} &&i^{\ast }F_{1}\ar[dr]_{}\ar[rr]^{}&& i^{\ast }F_{0 }\ar[rr]^{} &&i^{\ast }K \\
                &&&i^{\ast }K_{n-2}\ar[ru]^{} &&i^{\ast }K_{2 }\ar[ru]^{}&&i^{\ast }K_{1 }\ar[ru]^{} &&}
$$
where $i^{\ast }F_i\in i^{\ast }\mathcal{F}=i^{! }\mathcal{F}$, that is, $i^{\ast }K\in {(i^{! }\mathcal{F})}^\wedge_{n-1}$. So the $\mathbb{E}_\mathcal{A}$-triangle (\ref{want}) is exactly what we want. Similarly, for any $A\in\mathcal{A}$, we can prove that there exists an $\mathbb{E}_\mathcal{A}$-triangle
$$
  \xymatrix{ A\ar[r]^-{}&B\ar[r]^-{}&Z&},
$$
where $B\in i^{! }\mathcal{F}$ and $Z\in(i^{\ast }\mathcal{T})^{\vee}_{n-1}$.

(2) It is similar to (1).

(3) Note that $i_{\ast }i^{\ast }\mathcal{T}\subseteq \mathcal{T}$ and $j_{\ast }j^{\ast }\mathcal{F}\subseteq \mathcal{F}$, by Lemma \ref{impor}, we have  $\E_\mathcal{A}^{n+1}(i^\ast T, i^! F)\cong \E_\mathcal{B}^{n+1}( i_\ast i^\ast T, F)=0$ and $\E_\mathcal{C}^{n+1}(j^\ast T, j^\ast F)\cong \E_\mathcal{B}^{n+1}(T, j_\ast j^\ast F)=0$.
\qed
\begin{corollary}\rm
Assume that $(\mathcal{A},\mathcal{B},\mathcal{C})$ is a recollement of extriangulated categories with enough projectives and enough injectives. Assume that $(\mathcal{T},\mathcal{F})$ is a left (resp. right) $n$-cotorsion pair in $\mathcal{B}$. If $i^{\ast }$ and $i^{!}$ are exact, $j_{\ast }j^{\ast }\mathcal{F}\subseteq \mathcal{F}$, $i_{\ast }i^{\ast }\mathcal{T}\subseteq \mathcal{T}$ and $i_{\ast }i^{\ast }\mathcal{F}\subseteq \mathcal{F}$, then
\begin{itemize}

\item[\rm (1)]$(i^{\ast }\mathcal{T}, i^{!}\mathcal{F})$ is a left (resp. right) $n$-cotorsion pair in $\mathcal{A}$;

\item[\rm (2)] $(j^{\ast }\mathcal{T}, j^{\ast}\mathcal{F})$ is a left (resp. right) $n$-cotorsion pair in $\mathcal{C}$.
\end{itemize}
\end{corollary}

\begin{corollary}\rm
Assume that $(\mathcal{A},\mathcal{B},\mathcal{C})$ is a recollement of extriangulated categories with enough projectives and enough injectives.  Assume that $\mathcal X$ is an $(n+1)$-cluster tilting subcategory of $\mathcal{B}$. If $i^{\ast }$ and $i^{!}$ are exact, $j_{\ast }j^{\ast }\mathcal{X}\subseteq \mathcal{X}$, $i_{\ast }i^{\ast }\mathcal{X}\subseteq \mathcal{X}$ and $i_{\ast }i^{\ast }\mathcal{X}\subseteq \mathcal{X}$, then
\begin{itemize}

\item[\rm (1)]$(i^{\ast }\mathcal{X}, i^{!}\mathcal{X})$ is an $n$-cotorsion pair in $\mathcal{A}$;

\item[\rm (2)] $(j^{\ast }\mathcal{X}, j^{\ast}\mathcal{X})$ is an left $n$-cotorsion pair in $\mathcal{C}$.
\end{itemize}
\end{corollary}
\proof
This follows from Theorem \ref{main1} and  \cite[Theorem 3.5]{HZ}.
\qed
\begin{corollary}
{\rm In Theorem {\ref{main1}}, when $(\mathcal{A},\mathcal{B},\mathcal{C})$ is a recollement of abelian categories, it is just Theorem {\rm 3.4} in {\rm \cite{CWW}}}.
\end{corollary}

\section{From recollement of extriangulated categories to recollement of semi-abelian categories }
We first recall the notion of pseudo cluster tilting subcategory from \cite{HH}.

\begin{definition}\cite[Proposition 3.4]{HH}\label{pct}
 Let $(\mathcal C, \E, \s)$ be an extriangulated category  and $\mathcal X$ a subcategory of $\mathcal C$. $\mathcal X$ is called a pseudo cluster tilting subcategory of $\mathcal C$, if the following conditions are satisfied:
\begin{itemize}
\item[$(1)$] For any $C\in\mathcal C$, there exists an $\E$-triangle $\xymatrix@C=0.5cm{C\ar[r]^{a\;} & X_1 \ar[r]^{b} & X_2\ar@{-->}[r]^{\del}&,}$ where $X_1,X_2\in\mathcal X$ and $a$ is a left $\mathcal X$-approximation of $C$;

\item[$(2)$] For any $C\in\mathcal C$, there exists an $\E$-triangle $\xymatrix@C=0.5cm{X_3\ar[r]^{c\;} & X_4 \ar[r]^{d} & C\ar@{-->}[r]^{\eta}&,}$ where $X_3,X_4\in\mathcal X$ and $d$ is a right $\mathcal X$-approximation of $C$.
\end{itemize}

\end{definition}
Now we give some examples of pseudo cluster tilting subcategories. These examples come from \cite{HH}.

\begin{example}
\begin{itemize}
\item [\rm(a)] Let $Q$ be the quiver $\setlength{\unitlength}{0.03in}\xymatrix{1 \ar[r]^{\alpha} & 2}.$  Assume that $\tau_Q$ is the Auslander-Reiten translation in $D^b(kQ)$. We consider the repetitive cluster category $\C=D^b(kQ)/\langle \tau_Q^{-2}[2]\rangle$ introduced by Zhu in \cite{Z}, whose objects are the same in $D^b(kQ)$, and whose morphisms are given by $$\Hom_{D^b(kQ)/\langle \tau_Q^{-2}[2]\rangle}(X, Y)=\bigoplus_{i\in \mathbb{Z}}\Hom_{D^b(kQ)}(X, (\tau_Q^{-2}[2])^iY).$$  It is shown in \cite{Z} that $\C$ is a triangulated category.
We describe the Auslander-Reiten quiver of $\C$ in the following picture:
$$
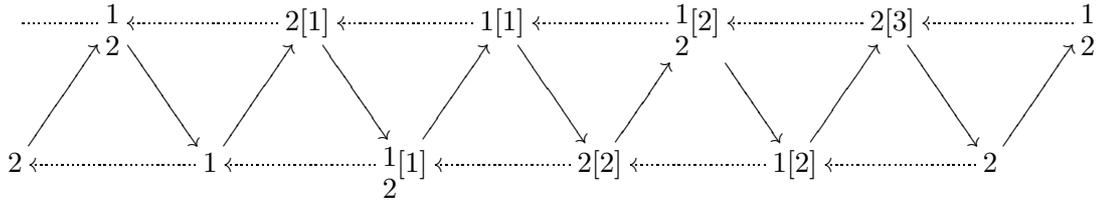
\xymatrix@!@C=0.4cm@R=1cm{  
&\txt{1\\2}\ar@{.}[l]\ar[dr]&&{2[1]}\ar@{.>}[ll]\ar[dr]&&1[1]
\ar@{.>}[ll]\ar[dr]&&{\txt{1\\2}[2]}\ar@{.>}[ll]\ar[dr]&&2[3]\ar@{.>}[ll]\ar[dr]&&\txt{1\\2}\ar@{.>}[ll]  \\
2\ar[ur] &&{1}\ar@{.>}[ll] \ar[ur]&&\txt{1\\2}[1]\ar@{.>}[ll]\ar[ur] && 2[2]\ar@{.>}[ll] \ar[ur]&&{1[2]}\ar@{.>}[ll] \ar[ur]&&2 \ar@{.>}[ll]\ar[ur]  }$$
It is straightforward to verify that the subcategory
$$\add(2\oplus1\oplus\begin{aligned}2\\[-3mm]1\end{aligned}\oplus 2[2]\oplus\begin{aligned}1\\[-3mm]2\end{aligned}[2])$$
is a pseudo cluster tilting subcategory of $\C$, but not a cluster tilting subcategory of $\C$, since ${\rm Hom}_{\C}(1,2[1])\neq 0$.

\item [\rm(b)]As a special  extriangulated category, let $(\C, \Omega)$ be an exact category with the exact structure $\Omega$. We consider the category of conflations of $\C$, denoted by $\Omega(\C)$. It is well known $\Omega(\C)$ is an additive category. The exact structure of $\Omega(\C)$ is the usual exact structure computed degree-wise, written as $(\Omega(\C),\Omega)$. Let $S(\C)$ be the full subcategory of $\Omega(\C)$ consisting of all splitting conflations. Then  $S(\C)$ is a pseudo cluster tilting subcategory of $(\Omega(\C),\Omega)$.

\end{itemize}
\end{example}

\begin{definition}\label{add recollement}{\rm \cite[Definition 2.1]{WL}}
Let $\mathcal{A}$, $\mathcal{B}$ and $\mathcal{C}$ be three additive categories. The diagram
$$
  \xymatrix{\mathcal{A}\ar[rr]|{i_{*}}&&\ar@/_1pc/[ll]|{i^{*}}\ar@/^1pc/[ll]|{i^{!}}\mathcal{B}
\ar[rr]|{j^{\ast}}&&\ar@/_1pc/[ll]|{j_{!}}\ar@/^1pc/[ll]|{j_{\ast}}\mathcal{C}}
$$
of additive functors is a recollement of additive category $\mathcal{B}$ relative to additive categories $\mathcal{A}$ and $\mathcal{C}$, if the following conditions satisfied:
\begin{itemize}
  \item [(R1)] $(i^{*}, i_{\ast}, i^{!})$ and $(j_!, j^\ast, j_\ast)$ are adjoint triples.
  \item [(R2)] $\Im i_{\ast}=\Ker j^{\ast}$.
  \item [(R3)] $i_\ast$, $j_!$ and $j_\ast$ are fully faithful.
\end{itemize}
\end{definition}

\begin{lemma}\label{hz}\rm{\cite[Theorem 3.9]{HH}}
Let $\mathcal{C}$ be an extriangulated categories and $\mathcal{T}$ a pseudo cluster tilting of $\mathcal{C}$. Then $\mathcal{C}/\mathcal{T}$ is a semi-abelian category.

\end{lemma}

\begin{lemma}\label{OO}
Let $(\mathcal{A}, \mathcal{B},\mathcal{C})$ be a recollement of extriangulated categories with enough projectives and enough injectives. ${i}^{*}$ and $ {i}^{!}$ are exact. If $\mathcal{T}$ is a pseudo cluster tilting subcategory of $\mathcal{B}$ and satisfies ${j_{\ast}j^{\ast}}\mathcal{T}\subseteq\mathcal{T}$, ${i_{\ast}i^{\ast}}\mathcal{T}\subseteq\mathcal{T}$. Then

$(1)$ ${j^{\ast}}\mathcal{T}$ is a pseudo cluster tilting subcategory of $\mathcal{C}$;

$(2)$ ${i^{\ast}}\mathcal{T}$ is a pseudo cluster tilting subcategory of $\mathcal{A}$.
\end{lemma}

\begin{proof} $(1)$ For any $C\in\mathcal{C}$, $j_{\ast}C\in\mathcal{B}$, since $\mathcal{T}$ is a cluster tilting subcategory of $\mathcal{B}$, then there exists an $\mathbb{E}_\mathcal{B}$-triangle $$\xymatrix{ j_{\ast}C\ar[r]^-{}&T_1\ar[r]^-{}& T_2\ar@{-->}[r]&}$$ with $T_1, T_2\in \mathcal{T}$. Because $j^{\ast}$ is exact and $j^{\ast}j_{\ast}\Rightarrow \Id_{\mathcal{C}}$ is a natural isomorphism, applying $j^{\ast}$ to the above $\mathbb{E}_\mathcal{B}$-triangle, we obtain an $\mathbb{E}_\mathcal{C}$-triangle $$\xymatrix{ C\ar[r]^-{}&j^{\ast}T_1\ar[r]^-{}& j^{\ast}T_2\ar@{-->}[r]&}$$ with $j^{\ast}T_1, j^{\ast}T_2\in j^{\ast}\mathcal{T}$.

Similarly, for any $C\in\mathcal{C}$, we have an $\mathbb{E}_\mathcal{C}$-triangle $$\xymatrix{ j^{\ast}T_3\ar[r]^-{}&j^{\ast}T_4\ar[r]^-{}&C \ar@{-->}[r]&}$$ with $j^{\ast}T_3, j^{\ast}T_4\in j^{\ast}\mathcal{T}$.

$(2)$ It is similar to $(1)$.
\end{proof}

Our third main result is the following.

\begin{theorem}\label{zh}
Let $\mathcal{A}$, $\mathcal{B}$ and $\mathcal{C}$ be three extriangulated categories with enough projectives and enough injectives. Assume that $\mathcal{B}$ admits a recollement relative to $\mathcal{A}$ and $\mathcal{C}$, i.e.
\begin{equation}
\xymatrix{\mathcal{A}\ar[rr]|{i_{*}}&&\ar@/_1pc/[ll]|{i^{*}}\ar@/^1pc/[ll]|{i^{!}}\mathcal{B}
\ar[rr]|{j^{\ast}}&&\ar@/_1pc/[ll]|{j_{!}}\ar@/^1pc/[ll]|{j_{\ast}}\mathcal{C}}.
\end{equation}
If ${i}^{*}$ and $ {i}^{!}$ are exact, $\mathcal{T}$ is a \emph{pseudo cluster tilting} subcategory of $\mathcal{B}$ and satisfies ${j_{\ast}j^{\ast}}\mathcal{T}\subseteq\mathcal{T}, {i_{\ast}i^{\ast}}\mathcal{T}\subseteq\mathcal{T}$. Then the semi-abelian category $\mathcal{B}/\mathcal{T}$ admits a diagram of additive functors relative to semi-abelian categories $\mathcal{A}/i^{\ast}\mathcal{T}$ and $\mathcal{C}/j^{\ast}\mathcal{T}$ as follows:
\begin{equation}\label{13}
\begin{array}{l}
\xymatrix{\mathcal{A}/i^{\ast}\mathcal{T}\ar[rr]|{\overline{i}_{*}}&&\ar@/_1pc/[ll]|{\overline{i}^{*}}\ar@/^1pc/[ll]|{\overline{i}^{!}}\mathcal{B}/\mathcal{T}
\ar[rr]|{\overline{j}^{\ast}}&&\ar@/_1pc/[ll]|{\overline{j}_{!}}\ar@/^1pc/[ll]|{\overline{j}_{\ast}}\mathcal{C}/j^{\ast}\mathcal{T}}.
\end{array}
\end{equation}
Moreover, the diagram \rm\ref{13} is a recollement of semi-abelian categories if and only if $\mathcal{T}\subseteq\Ker{j^{*}}$.
\end{theorem}
\proof
By Lemma \ref{OO} and Lemma \ref{hz}, we have that $\mathcal{B}/\mathcal{T}, \mathcal{C}/j^{\ast}\mathcal{T}$ and $\mathcal{A}/i^{\ast}\mathcal{T}$ are semi-abelian categories. The remaining proofs are similar to the proofs of
\cite[Theorem 4.5]{HHZ}, we omit it.

\qed

\textbf{Jian He}\\
Department of Applied Mathematics, Lanzhou University of Technology, 730050 Lanzhou, Gansu, P. R. China\\
E-mail: \textsf{jianhe30@163.com}\\[0.3cm]
\textbf{Jing He}\\
School of Science, Hunan University of Technology and Business, 410205 Changsha, Hunan P. R. China\\
E-mail: \textsf{jinghe1003@163.com}


\begin{thebibliography}{99}
\addtolength{\itemsep}{-0.5em}

\bibitem{BBD} A. Beilinson, J. Bernstein and P. Deligne, Faisceaux pervers, in: Analysis and topology on singular spaces, {I}, Luminy, 1981, {Ast\'erisque}, {vol. 100}, Soc. Math. France,
  Paris, 5--171,  1982.
\bibitem{CWW} W. Cao, J. Wei and K. Wu. Recollement and $n$-cotorsion pairs. arXiv: 2403.04220.



\bibitem{HMP} M. Huerta, O. Mendoza and M. A. P\'{e}rez. $n$-Cotorsion pairs,\emph{J. Pure Appl. Algebra} \textbf{225}(5) (2021) 106556.





\bibitem{HH} J. He and J. He, Semi-abelian categories arising from pseudo cluster tilting subcategories, arXiv: 2309.04075, 2023.

\bibitem{HHZ} J. He, Y. Hu and P. Zhou, Torsion pairs and recollements of extriangulated categories,
\emph{Comm. Algebra} \textbf{50}(5) (2022) 2018--2036.

\bibitem{HZ} J. He and P. Zhou, On the relation between $n$-cotorsion pairs and $(n+1)$-cluster tilting subcategories, \emph{J. Algebra Appl.} \textbf{21}(1) (2022) 12 pages.

\bibitem {HZZ} J. Hu, D. Zhang and P. Zhou, Proper classes and Gorensteinness in extriangulated categories,
 \emph{J. Algebra}  \textbf{551} (2020) 23--60.


\bibitem{I}  O. Iyama, Cluster tilting for higher Auslander algebras, \emph{Adv. Math} \textbf{226} (2008) 1--61.


\bibitem{LN} Y. Liu, H. Nakaoka, Hearts of twin cotorsion pairs on extriangulated categories,  \emph{J. Algebra} \textbf{528} (2019) 96--149.


\bibitem{MXZ} X. Ma Z. Xie and T. Zhao, Support $\tau$-tilting modules and recollement, \emph{Colloq. Math.} \textbf{167}(2) (2022) 303--328.

\bibitem{MV} R. MacPherson and K. Vilonen, Elementary construction of perverse sheaves, \emph{Invent. Math.} \textbf{84}(2) (1986) 403--435.

\bibitem{NP} H. Nakaoka and Y. Palu, Extriangulated categories, Hovey twin cotorsion pairs and model structures, \emph{Cah. Topol. G\'{e}om. Diff\'{e}r. Cat\'{e}g.} \textbf{60}(2) (2019) 117--193.

\bibitem{WL} M. Wang and Z. Ling, Recollement of additive quotient categories, arXiv: 1502.00479, 2015.
\bibitem{WWZ} L. Wang, J. Wei and H. Zhang, Recollements of extriangulated categories, \emph{Colloq. Math.} \textbf{167}(2) (2022) 239--259.

\bibitem{Z} B. Zhu, Generalized cluster complexes via quiver representations,
\emph{J. Algebraic Combin.} \textbf{27}(1)(2008) 35--54.



\bibitem{ZZ} P. Zhou and B. Zhu, Triangulated quotient categories revisited, \emph{J. Algebra}  \textbf{502} (2018) 196--232.





\end{thebibliography}
\end{document}